\documentclass[11pt]{article}

\def\UseSection{%%
        \numberwithin{equation}{section}
        \newtheorem{theorem}    {Theorem}[section]
        \DefineTheorems % Use this to define other environments to be
}
\usepackage[textwidth=500pt,textheight=650pt,centering]{geometry} % 11pt

\usepackage{amsfonts}
\usepackage{amsmath,amssymb,amsthm}
\usepackage{appendix}
\usepackage{bbm} % used in \newcommand{\1}{\mathbbm{1}}
\usepackage{amsbsy}
\usepackage{enumerate}
\usepackage{cite}
\usepackage{showkeys-hara}

\usepackage[bookmarks, colorlinks, breaklinks,
pdftitle={},pdfauthor={}]{hyperref}

\hypersetup{
  linkcolor=black,
  citecolor=black,
  filecolor=black,
  urlcolor=black}

\usepackage{verbatim}
\usepackage{tikz}
\usepackage{color}

\newcommand{\black}{\black}

\usepackage[font=small,labelfont=bf]{caption}

\numberwithin{equation}{section}

\usetikzlibrary{calc,decorations.markings}
\usetikzlibrary{decorations.pathmorphing}
\usetikzlibrary{decorations.pathreplacing}
\usetikzlibrary{patterns}
\usetikzlibrary{arrows}

\newcommand{\bb}[1]{\mathbb{#1}}

\newcommand{\blank}[1]{}

\newcommand{\E}{\bb E}
\newcommand{\R}{\bb R}
\newcommand{\Z}{\bb Z}

\newcommand{\T}{\bb T}
\renewcommand{\P}{\bb P}

\newcommand{\Acal}   {\mathcal{A}}

\newcommand{\Ccal}   {\mathcal{C}}

\newcommand{\Ecal}   {\mathcal{E}}

\newcommand{\Wcal}   {\mathcal{W}}

\newcommand{\nnb}	{\nonumber \\}

\def\DefineTheorems{%%
	\newtheorem{lemma}      [theorem] {Lemma}

	\newtheorem{prop}        [theorem] {Proposition}
	\theoremstyle{definition}% ``defn'' theorem style
	\newtheorem{defn}       [theorem] {Definition}
	\newtheorem{rk}       [theorem] {Remark}
}

\UseSection
\newcommand{\eps}{\varepsilon}

 \title {
   The scaling limit of the weakly self-avoiding walk on a high-dimensional torus
 }

 \author{
    Emmanuel Michta\thanks{Department of Mathematics,
     University of British Columbia,
     Vancouver, BC, Canada V6T 1Z2.
     Michta: \url{https://orcid.org/0000-0001-7222-0422}, {\tt michta@math.ubc.ca}.}
}

  \date{\vspace{-5ex}} %for arXiv

\begin{document}
\maketitle

\begin{abstract}
We prove that the scaling limit of the weakly self-avoiding walk on a $d$-dimensional discrete torus is Brownian motion on the continuum torus if the length of the rescaled walk is $o(V^{1/2})$ where $V$ is the volume (number of points) of the torus and if $d>4$. We also prove that the diffusion constant of the resulting torus Brownian motion is the same as the diffusion constant of the scaling limit of the usual weakly self-avoiding walk on $\Z^d$. 
%This is another way, following \cite{MS22}, to see 
This provides further manifestation of the fact that the weakly self-avoiding walk model on the torus does not feel that it is on the torus \emph{up until} it reaches about $V^{1/2}$ steps which we believe is sharp.
\end{abstract}
\section{Introduction and results}

\subsection{Introduction}
A self-avoiding walk (SAW) on any graph is a path on this 
graph that does not visit any vertex more than once. The SAW-model consists in understanding properties of the SAW : How does the number of SAW of length $n$ behaves as a function of $n$? How far does a uniformly chosen length-$n$ SAW go? What is the scaling limit of such walks and how can it be proven to exist? etc. Historically and still today, most of the research on this topic 
focuses on the case where the graph is a lattice and especially the euclidean lattice $\Z^d$ where the 
dimension $d$ plays a key role in the behaviour of this model. In dimension $d>4$ the model is now very well understood following the introduction 
of the \emph{lace-expansion} by Brydges and Spencer 
in 1985 \cite{BS85} in the context of the closely related \emph{weakly self-avoiding walk} (WSAW). % and its further 
%developments by Hara and Slade.  
Hara and Slade developed the lace-expansion to prove that 
the usual self-avoiding walk model 
has \emph{mean-field} behaviour in dimension $d \geq 5$ in the early 90's \cite{HS91,HS92a,HS92b} which means in short that the SAW behaves asymptotically like the simple random walk. It was shown in this context \cite{HS92a,Slad89} that the scaling limit of the SAW on $\Z^d$ is Brownian 
motion with some diffusion constant $D>1$. The scaling limit is also believed to be Brownian motion at the critical dimension $d_c =4$ but with extra $\log$ corrections to the rescaling factors \cite{Dupl86}. Relatively recent rigorous results in that direction have been obtained through renormalization group arguments for the WSAW in \cite{BBS-saw4,BBS-saw4-log}. In lower dimensions the picture is quite different. In two dimension, the scaling limit of the SAW is conjecturally the so-called Schramm-Loewner Evolution with parameter $\kappa = 8/3$. 
%SLE$_{8/3}$. 
And in fact this is known to be the case under the widely open assumption that the scaling limit exists and is 
conformally invariant following the breakthrough work of Lawler, Schramm and Werner \cite{LSW04}. In three dimensions there seems to be no good candidate as of today for the scaling limit and the model more generally remains poorly understood even from the physics standpoint.

Here we study the WSAW model where true SAW are assigned weight $1$ and other walks get penalized for intersecting and have a small but non-zero weight. A more formal definition is given in the next subsection. We are interested in the model on a $d$-dimensional torus (a box with periodic boundary conditions) of sidelength $r\geq 3$ where $d > 4$. It was pointed out and partly proven in \cite{MS22} that the WSAW on the torus behaves the same as its $\Z^d$ counterpart provided that its length is less than $r^{d/2}$ and that this should be sharp. In this article we study the scaling limit of the WSAW on the torus which exhibits again this interplay between $r$ and $n$ the length of the WSAW when both go to infinity. We show that the scaling limit of the WSAW is Brownian motion on the continuum torus if the walk that we rescale has length less than $r^{d/2}$ and more than $r^2$, which is possible only if $d>4$. If the walk has length much less than $r^2$ it is not possible to scale both the discrete torus into a continuum one and the walk into a continuous non-degenerate random curve and the model exhibits no interesting behaviour. Our main tool to prove the above is the lace expansion. Since its introduction in \cite{BS85}, the lace expansion technology has been simplified and explained by several authors \cite{MS93,Slad06,BHK18} (to which we refer for further background) as well as applied to a broad range of models (percolation, Ising, $\phi^4$, etc.) above their critical dimensions.

\subsection{Notation}

We write $f \sim g$ to mean $\lim f/g =1$, $f\prec g$ to mean $f \leq c_1 g$ with $c_1>0$ and $f \succ  g$ to mean  $g \prec  f$. We also write $f \asymp g$ when $ g \prec f \prec g$. We also write $ f \ll g$ if $f/g = o(1)$ and $f \gg g$ if $ g \ll f $. For any vector $u \in \R^d$, $|u|$ denotes the $L^2$ norm of $u$.
Constants are permitted to depend on $d$.

\subsection{The weakly self-avoiding walk}

An $n$-step walk on $\Z^d$ is a function $\omega :\{0,1,\ldots,n\} \to \Z^d$ with $\|\omega(i)-\omega(i-1)\|_1=1$ for $1 \le i \le n$.
We let $\Wcal_n$ denote the set of $n$-step walks starting at $0$.
For an $n$-step walk $\omega$, and for $0 \leq s < t \leq n$, we define
\begin{equation}
\label{e:Ustdef}
    U_{st}(\omega) 
    =
    \begin{cases}
        -1 & (\omega(s)=\omega(t)) \\
        0 & (\omega(s)\neq \omega(t))
    \end{cases}
\end{equation}
and 
\begin{equation}
\label{e:K_def}
	K[0,n] = \prod_{0 \leq s < t\leq n}(1+\beta U_{st}(\omega))
\end{equation}
where we keep the dependency on $\omega$ and $\beta$ implicit to simplify the notation. Given $\beta \in [0,1]$, we define the
\emph{partition function} $c_n = c_n(\beta)$ by
\begin{equation}
\label{e:cndef}
    c_n = \sum_{ \omega \in \Wcal_n } K[0,n].
\end{equation}
The product in \eqref{e:K_def}
discounts $\omega$ by a factor
$1-\beta$  for each pair $s,t$ with an intersection $\omega(s)=\omega(t)$.
When $\beta=0$, $c_n$ is simply the number of $n$-step walks and is thus equal to $(2d)^n$. For $\beta=1$, $c_n$ is the number of $n$-step strictly self-avoiding walks.
The case $\beta \in (0,1)$ is the \emph{weakly self-avoiding walk}.
In particular, it is proved in \cite{HHS98}
that for $d>4$ with $\beta>0$ sufficiently small
\begin{equation}
\label{e:cnasy}
    c_n = A\mu^n (1+O(\beta n^{-(d-4)/2})),
\end{equation}
with $A = 1+O(\beta)$ and where $\mu = \mu(\beta,\Z^d)$ is the WSAW \emph{connective constant of $\Z^d$}. In addition, the \emph{mean-square displacement} defined by the left-hand side of \eqref{e:mean_square_disp} satisfies
\begin{equation}
\label{e:mean_square_disp}
	\frac{1}{c_n}\sum_{\omega \in \Wcal_n}|\omega(n)|^2 K[0,n] = Dn(1+O(n^{-1/2}))
\end{equation}
where $D = D(\beta)>1$ is some constant called the \emph{diffusion constant} of the WSAW. It is also known (see \cite[Theorem 6.1.8]{MS93}) that the scaling limit of the weakly self-avoiding walk is Brownian motion in dimension $d>4$ and that the diffusivity of the limiting Brownian motion is equal to $D$.
It is widely acknowledged even though conjectural that the asymptotic behaviour of the WSAW  and the SAW are the same in all dimensions and thus that both models lie in the same universality class.

There is a natural probability measure on $\Wcal_n$ denoted $\P_{\beta, n}$ that comes with the WSAW model. If we let $\E_{\beta,n}$ be its corresponding expectation, then it is defined by 
\begin{equation}
	\E_{\beta,n}\,f = \frac{1}{c_n}\sum_{\omega \in \Wcal_n}f(\omega)K[0,n]
\end{equation}
for any  bounded function $f\colon \Wcal_n \to \R$. We write $\omega \sim \beta$-WSAW$_n$ to mean that $\omega$ has law $\P_{\beta,n}$.

We are mostly interested in walks on a discrete torus $\T_r^d = (\Z/r\Z)^d$ for an integer $r \ge 3$ which will be taken to be large.
The \emph{volume} of the torus is $V=r^d$. Torus walks are defined as on $\Z^d$ but with steps $\omega(i)-\omega(i-1)$
computed using addition modulo $r$ in each one of the $d$ components. We let $\Wcal_n^\T$ denote the set of $n$-steps torus walks starting at $0$. In the same way as on $\Z^d$, we define torus quantities
\begin{align}
\label{e:cnTdef}
    c_n^\T 
    = \sum_{\omega\in \Wcal_n^\T} K[0,n],\qquad
    \E_{\beta,n,r}^\T \,g 
    = \frac{1}{c_n^\T}\sum_{\omega\in \Wcal_n^\T}g(\omega) K[0,n],
\end{align}
for any $g\colon \Wcal_n^\T \to \R$. We write $\omega \sim \beta$-WSAW$^{\T}_{n,r}$ to mean that $\omega$ has law $\P^\T_{\beta,n,r}$.

\subsection{Main result}
We introduce some notations before rescaling the walk and stating our main result. Let $\T^d = (\R/\Z)^d$, $T>0$ possibly infinite and $\Ccal^0_0([0,T],\T^d)$ be the space of continuous functions from $[0,T]$ to $\T^d$ starting from $0$ and if $T=\infty$ then we instead take $\Ccal^0_0([0,\infty),\T^d)$.
For a torus walk $\omega \in \Wcal_n^\T$ with $n \ll V^{1/2}$ and $n/r^2 \uparrow T$ ($T=\infty$ is allowed), we define the mapping $x^{(n)}$ from $\Wcal^\T_{n}$ to $\Ccal^0_0([0,T],\T^d)$ which is the proper rescaling of the torus walk by 
\begin{align}
\label{e:x_def}
	x^{(n)}_{k/r^2}(\omega) 
	&= \frac{\omega(k)}{r} \text{ for } k \in \{0,1, \cdots, n \} ,
\end{align}
where we see $\frac{\omega(k)}{r}$ as an element of $\T^d$ rather than $r^{-1}\T_r^d$ by a canonical injection. We then let $x_t^{(n)}(\omega)$ interpolate linearly between the above values up to $t = n/r^2$ and $x_t = x_{n/r^2}$ for $t \in [n/r^2,T]$ (respectively in $[n/r^2,\infty)$ if $T=\infty$). 
We define the stopped process $B^T = (B_t)_{0\leq t \leq T}$ where $B$ is an $\R^d$-Brownian motion defined on some probability space $(\Omega,\Acal,\P)$ satisfying
\begin{equation}
	\E\,|B_1|^2 = D.
\end{equation} 
So the diffusion constant of all the Brownian motions appearing in what follows is taken equal to $D$ which is defined in \eqref{e:mean_square_disp}. 
In the case $T = \infty$, $B^T$ is simply $B$.
We defer the precise definition of convergence in law and continuity to Section~\ref{sec:set-up}. Our main result is the following:

\begin{theorem}
\label{thm:main}
	Let $d>4$, $\beta$ small enough. If $\omega \sim \beta$-{\rm WSAW}$^\T_{n,r}$ with $n = n_r$ satisfying $n \ll V^{1/2}$ and $n/r^2 \uparrow T>0$ (possibly $\infty$) then $X^{(n)}= x^{(n)}(\omega)$ converges in law to a Brownian motion on the torus with diffusion constant $D$ and length $T$. And $D$ is the usual $\Z^d$-diffusion constant defined in \eqref{e:mean_square_disp}. In other words, for any continuous and bounded functional $f$ on $\Ccal^0_0([0,T],\T^d)$ we have
\begin{equation}
	\E_{\beta,n,r}^\T\, f(X^{(n)}) 
	\to 
	\E\,f(B^T \hspace{-0.25cm}\mod 1) \text{ as } r \to \infty.
\end{equation}
\end{theorem}

Theorem \ref{thm:main} follows easily from the classical fact that the scaling limit of the strictly self-avoiding walk is Brownian motion (see \cite[Theorem 6.1.8]{MS93}) with diffusion constant $D$ and from the following more recent result proven in \cite{MS22} which itself carries more technicality than the proof of Theorem \ref{thm:main}.
\begin{theorem}
\label{thm:dilute}
For $d>4$ and $C_0>0$, if $\beta >0$ is sufficiently small and
if
$n \le  C_0 V^{1/2}$, then
\begin{equation}
\label{e:cnmr}
    c_{n}^{\T} = A\mu^n
    \left[ 1+O(\beta)\Big( \frac{1}{n^{(d-4)/2}} + \frac{n^2}{V}\Big) \right]
    ,
\end{equation}
where $A$ is the same constant as in \eqref{e:cnasy} and the error term depends on $C_0$ but not on $n,V,\beta$. In particular $c_n \sim c^\T_n$ for $n=o(V^{1/2})$ going to infinity.
\end{theorem}

 Since we are considering the WSAW rather than the usual (strictly) SAW, the results in \cite{MS93} do not apply directly but as we discuss above Lemma \ref{lem:fdd} it is elementary to adapt them to the easier case of the WSAW.

\subsection{Discussion}

In fact we shall only prove the case $T=\infty$ which is the most interesting and from which the proof can be easily adapted
 to cover the easier case $T>0$ and finite. We exclude the case $T = 0$ because then the rescaling in \eqref{e:x_def} is 
 degenerate in the sense that the torus is properly scaled whereas the rescaled walk shrinks to $0$ as is noted in Remark \ref{rmk:degener_lim}.  
 
 We refer to the regime of values of $n$ for which the limiting behaviour (such as the scaling limit) of the torus WSAW is the same as on $\Z^d$ as the subcritical regime or dilute 
phase. It corresponds conjecturally to walks having length $o(V^{1/2})$. On the other hand it would be interesting to determine what the scaling limit of the walk is when its length is of order $V^{1/2}$ (critical regime) or much larger
  than $V^{1/2}$ (supercritical regime or dense phase). In the critical case we expect the scaling limit to be Brownian motion with a different diffusion constant than on $\Z^d$. The above 
  conjectural distinction of regimes : $n \ll V^{1/2}$, $n\asymp V^{1/2}$ and $n \gg V^{1/2}$ has been first identified 
  for the simpler and solvable case of the SAW on the 
  complete graph in \cite{Slad20} (see also \cite{MS22,Slad22} for progress on the torus and the hypercube). 
  In order to establish a version of Theorem~\ref{thm:main} for the critical regime we would need at least to understand better how $c_n^\T$ behaves when $n \asymp V^{1/2}$. 
In this regime we do not expect $c_n^\T$ to be equivalent to $c_n$ when $n \to \infty$. This  prevents the proof contained in this article to apply directly in this case and is a manifestation of the fact that new ideas are required to understand the critical regime as was already noted in \cite[Remark 5.3]{MS22}. 

We end this subsection by discussing a closely related open problem. Recall that the torus and $\Z^d$ WSAW two-point functions are given by 
\begin{equation}
	G^{*}_z(x) = \sum_{\omega \in \Wcal^*(x)}z^{|\omega|}K[0, |\omega|]
\end{equation}
where $`` * "$ is either replaced by $\T$ for torus walks or nothing for $\Z^d$ walks. The sum is over torus or $\Z^d$ walks starting from $0$ and ending at $x$. The torus and $\Z^d$ susceptibilities are defined similarly by $\chi^*(x) = \sum_x G^*_z(x)$ and the sum is either over $x$ in $\T_r^d$ or in $\Z^d$. The \emph{critical window} corresponds to the $r$-dependent values of $z$ for which $\chi^\T(z) \asymp V^{1/2}$. We let $z_c$ denote the \emph{critical point} for the $\Z^d$-WSAW model i.e. the radius of convergence of $\chi(z)$.  The proof of Theorem \ref{thm:dilute} uses the key estimate of \cite[Theorem 1.2]{Slad20_wsaw} called the \emph{plateau} of the torus two-point function critical window (see \cite{Slad20,MS22,Slad20_wsaw,GEZGD17, DGGZ22} for background on the plateau, the critical window and simulations in the case of SAW). The plateau phenomenon states that inside the critical window $G^\T_{z}(x)$ decays like $G_{z_c}(x) \asymp |x|^{2-d}$ (when $d>4$) for small values of $x$ up until it reaches a constant value of order $V^{-1/2}$ where it levels off for larger values of $x$.
Some limitations encountered in \cite{MS22} come from the fact that the plateau is not known to hold inside the whole critical window.
It would thus be a natural first step to find a way to extend the plateau of \cite[Theorem 1.2]{Slad20_wsaw} to the entire critical window. This would lead to a better understanding of the critical regime for WSAW on the torus. We note that similar plateaux have been obtained for the simple random walk \cite{Slad20_wsaw, ZGFDG18, ZGDG20} when $d>2$, for percolation on a high-dimensional torus \cite{HMS22} and in part for Ising \cite{Papa06,ZGFDG18, ZGDG20} when $d>4$.
  
\section{Scaling limit on the torus}
\subsection{Set-up}
\label{sec:set-up}
For any $r\geq 3$, we recall that $\T_r^d$ is the $d$-dimensional discrete $r$-torus, that is  $\T_r^d = (\Z/r\Z)^d$. We let $\T^d = (\R/\Z)^d$ be the continuum torus. We will often see the discrete $r$-torus and the continuum torus as  $[-r/2,r/2)\cap \Z^d$ with addition modulo $r$ and as $[-1/2,1/2)$ with addition modulo $1$ respectively. For $x\in \T^d$ we denote by $(x)_1$ (for $x\in \T_r^d$ we denote by $(x)_r$ respectively) the unique representative in the equivalence class of $x$ that is inside $[-\frac12, \frac12)$ (inside $[-\frac{r}{2}, \frac{r}{2})$ respectively). This is useful to canonically embed $\T_r^d$ into $\Z^d$ and $\T^d$ in $\R^d$. In the rest of the article we will mention convergence in law on several metric spaces that we now define.  We let	
\begin{itemize}
\item $\Ccal^{0}(\R^d) = (\Ccal^0_0([0,\infty),\R^d),\|\cdot\|_\infty)$ be the metric space of continuous function $x$ from $[0, \infty)$ to $\R^d$ with $x(0) = 0$ endowed with the topology of uniform convergence. By abuse of notation we shall also denote by $\Ccal^{0}(\R^d)$ the corresponding measurable space with its associated Borel $\sigma$-algebra. 
\item $\Ccal^{0}(\T^d) = (\Ccal_0^0([0,\infty),\T^d),\|\cdot\|_\infty)$ is defined similarly with $\|x\|_\infty = \sup_{t \geq 0}{|(x(t))_1|}$ and the corresponding topology and Borel $\sigma$-algebra.
\end{itemize}
We denote by $\Rightarrow$ the convergence in law (see \cite[Chapter 1]{Bill95}) of probability measures on any of the above metric spaces without distinction since it shall be clear from context where the convergence holds.

In order to compare walks in $\T_r^d$ and in $\Z^d$ we introduce \emph{the lift $\hat \ell$ of a
torus walk} which is to be thought of as the unwrapping to $\Z^d$ of a walk on $\T_r^d$. The occurence of the hat on top of the $\ell$ is to make explicit the distinction between the lift of walks and the lift of processes which is introduced in Proposition \ref{prop:lift_hom}.
\begin{defn}[The $\Z^d$-lift of torus walks]
We define the \emph{lift} $\hat \ell$ of a torus walk by the operator

\begin{align} 
	\hat \ell\colon \Wcal_n^\T &\rightarrow \Wcal_n \\
	 \omega &\mapsto \hat \ell[\omega]\nonumber
\end{align}
which to any $\omega \in \Wcal_n^\T$ associates the $\Z^d$-walk $\hat \ell[\omega]$ defined by $\hat \ell[\omega](0)=0$ and
\begin{equation}
\label{e:liftdef}
	\hat \ell[\omega](k) = \hat \ell[\omega](k-1) + (\omega(k) - \omega(k-1))_r \qquad (1\leq k \leq n).
\end{equation}
Due the nearest-neighbour constraint, when $r \geq 3$ the map $\hat \ell$ is a bijection from $\Wcal_n^\T$ onto $\Wcal_n$. 
\end{defn}

This bijection has the following consequence that $c_n^{\T}$ in \eqref{e:cnTdef} can be rewritten as a sum over $\Z^d$-walks  where an intersection occur if the walk visits two points that are equal modulo $r$ (in every coordinate). This gives
\begin{equation}
\label{e:cn_T_other_def}
	c_n^\T = \sum_{\omega \in \Wcal_n}K^\T[0,n],\quad \, \text{ where} \qquad  K^\T[0,n] = \prod_{0 \leq s < t \leq n}(1+U^\T_{st}(\omega))
\end{equation}
and with $U_{st}^\T$ defined by
\begin{equation}
	    U^\T_{st}(\omega) =
    \begin{cases}
        -1 & (\omega(s)=\omega(t) \mod r)
        \\
        0 & (\omega(s)\neq \omega(t) \mod r).
    \end{cases}
\end{equation}

Similarly we extend the definition of the lift to the setting of processes instead of walks but since it requires some additional yet elementary work we make it a proposition. The following proposition is useful because it allows us to partially convert the problem of the scaling limit on the torus to a similar problem on $\Z^d$ where the theory is already developed.
\begin{prop}[The $\R^d$-lift of torus processes]
\label{prop:lift_hom}
We define the \emph{lift} of a process to be the linear operator
\begin{align} 
	\ell\colon\Ccal^{0}(\T^d) &\rightarrow \Ccal^{0}(\R^d) \\
	  x &\mapsto \ell[x]\nonumber
\end{align}
which to any function $x \in \Ccal^{0}(\T^d)$ associates $y = \ell[x]$ the unique continuous $\R^d$ function $(y(t))_{t \geq 0}$ starting at $0$ such that $x(t) = y(t) \mod 1$ for all $t\geq 0$. Furthermore the lift $\ell$ of a torus process is a linear homeomorphism of $\Ccal^{0}(\T^d)$ onto its range $\Ccal^{0}(\R^d)$.
\end{prop}

\begin{proof}
The proof consists in formalizing the obvious.
By construction we shall show that the lift of a torus process is well defined and is unique. We recall that $(x)_1$ is the unique element in the equivalence class of $x$ that is inside $[-\frac12, \frac12)^d$. Let $x \in \Ccal^{0}(\T^d)$, we construct $y \in \Ccal^{0}(\R^d)$ as follows. Fix $y(0) = x(0) = 0$ and $t_0 = 0$. Consider the sequence $(t_n)$ defined by
\begin{equation}
\label{e:seq_lift}
	t_{n+1} = \inf\{t>t_n,\; |(x(t)-x(t_{n}))_1| = 1/8\}
\end{equation}
and $t_{n+1}= \infty$ if the above set is empty. One sees from the continuity of $x$ that $t_{n+1}>t_n$ and $t_n \to \infty$ (or $= \infty$ for some $n$). Then we set
\begin{equation}
\label{e:seq_lift_2}
	y(t) = y(t_n) + (x(t)-x(t_n))_1 \qquad (\text{for }t\in (t_{n},t_{n+1}] \text{ and }n \geq 0)
\end{equation}
and define $\ell[x] = y$ which is in $\Ccal^0(\R^d)$ by construction. The fact $x(t) = y(t) \mod 1$ is clearly satisfied and can be checked on $(t_n,t_{n+1}]$ by induction on $n$. To show the uniqueness suppose that $y_1,y_2 \in \Ccal^0(\R^d)$ both satisfy $x(t) = y_i(t) \mod 1$ for all $t \geq 0$ and $i=1,2$. Then $y_1(t) - y_2(t) = 0 \mod 1$ so that $y_1(t) - y_2(t) = u(t)$ with $u(t) \in \Z^d$ and by continuity of both $y_1$ and $y_2$ and the fact that $y_1(0) = y_2(0) =0 $ this implies that $u=0$. 
The fact that the range of $\ell$ is $\Ccal^0(\R^d)$ follows from the fact that for any $y \in \Ccal^0(\R^d)$, the canonical projection on the torus $y \mod 1$ is an element of $\Ccal^0(\T^d)$ which satisfies $\ell[ y \mod 1] = y$. This also settles the fact that $\ell^{-1}$ is just the canonical projection onto the torus. The continuity of $\ell^{-1}$ is clear from the fact that it is a projection. The linearity of $\ell$ then follows from the bijectivity of $\ell$ and the linearity of the canonical projection. 
Note that for any $\eps>0$ with $\eps <1/8$ we have that $\|x\|_{\infty}<\eps$ implies that $ \ell[x](t) = (x(t))_1 \text{ for all }t \geq 0 $
and thus that $\|\ell[x]\|_\infty < \eps $ which proves the  continuity of $\ell$. 
\end{proof}

We let $B$ be the usual Brownian motion on $\R^d$ and define $B^\T$, the \emph{Brownian motion on the torus}, by

\begin{equation}
\label{e:torus_BM}
	B^\T = \ell^{-1}(B) = B\hspace{-0.2cm}\mod 1.
\end{equation}

\subsection{Reduction of proof}
\label{sec:scaling_limit}
In this section we reduce the proof of Theorem \ref{thm:main} to a convergence in law on $\Z^d$ by lifting the torus walks and their corresponding rescaled processes.
We focus on the case $T = \infty$ which means that $n/r^2 \uparrow \infty$ because it is the most interesting case. The proof for an arbitrary finite and positive $T$ is completely analogous and requires no additional effort and the case of $T=0$ is settled in Remark~\ref{rmk:degener_lim}. From now on $B$ is a $d$-dimensional Brownian motion on some probability space $(\Omega, \Acal, \P)$ satisfying 
\begin{equation}
	\E |B_1|^2 = D
\end{equation} 
where we recall that $D$ is the diffusion constant defined in \eqref{e:mean_square_disp}. We define $B^\T$ as in \eqref{e:torus_BM}. 
We want to show that whenever $\omega \sim \beta$-WSAW$^\T_{n,r}$ and $X^{(n)} = x^{(n)}(\omega)$ then for $r^2 \ll n \ll V^{1/2}$
\begin{align}
	(X^{(n)}_t)_{t \geq 0} \Rightarrow (B^\T_t)_{t \geq 0}.
\end{align}
We will do so by studying instead the lift of torus walks which is equivalent from the following lemma. 
\begin{lemma}[Studying the lift of a process is enough]
\label{lem:lift_equiv}
	Let $P,P_1,P_2, \cdots $ be probability measures on $\Ccal^{0}(\T^d)$. Then the following equivalence holds
	\begin{equation}
	 P_n \Rightarrow P 
	\quad\text{ if and only if }\quad
	P_n \circ \ell^{-1} \Rightarrow P \circ \ell^{-1}.
	\end{equation}
\end{lemma}
\begin{proof}
	This is an easy consequence of the \emph{mapping theorem} (see \cite[Theorem 2.7]{Bill95}) for convergence in law together with the fact that $\ell$ is a homeomorphism from $\Ccal^{0}(\T^d)$ to $\Ccal^0(\R^d)$ by Proposition \ref{prop:lift_hom}.
\end{proof}

We define the corresponding $\Z^d$ $\beta$-WSAW renormalized in a similar way as $X$ except that it is thus an $\R^d$ process. In detail, for $\omega$ a walk in $\Wcal_n$ we define the mapping $y^{(n)}$ from $\Wcal_{n}$ to $\Ccal^0(\R^d)$ by
\begin{align}
\label{e:y_def}
	y^{(n)}_{k/r^2}(\omega)
	&= \frac{\omega(k)}{r} \text{ for } k \in \{0, \cdots, n\} 
\end{align} 
and $y^{(n)}_t$ interpolates linearly between these values and also
\begin{equation}
	y^{(n)}_{t} =  y^{(n)}_{n/r^2} \text{ for }t \geq n/r^2.
\end{equation}
From \eqref{e:x_def} and \eqref{e:y_def} we see that $y^{(n)}$ is nothing else than
\begin{equation}
\label{e:lift_comp}
	y^{(n)} = \ell \circ x^{(n)} \circ \hat \ell^{-1}.
\end{equation}

\begin{lemma}
\label{lem:implication_conv}
Let $\beta>0$ be small enough, $n = o(V^{1/2})$, $\omega \sim \beta $-{\rm WSAW}$_n$, $\omega^\T \sim \beta$-{\rm WSAW}$^\T_{n,r}$ with $X^{(n)} = x^{(n)}(\omega^\T)$, $Y^{(n)} = y^{(n)}(\omega)$ and suppose that 
\begin{equation}
\label{e:full_zd_conv}
	(Y^{(n)}_t)_{t \geq 0} \Rightarrow (B_t)_{t \geq 0}
\end{equation}
holds. Then we have 
\begin{equation}
	(X^{(n)}_t)_{t \geq 0} \Rightarrow (B^\T_t)_{t \geq 0}.
\end{equation}
\end{lemma}

\begin{proof}
	We let $\beta>0$ be small enough and $X$, $Y$ as in the statement of Lemma \ref{lem:implication_conv}. By assumption for any continuous bounded functional $f$ on $\Ccal^{0}(\R^d)$ we have
\begin{equation}
	\E_{\beta,n}\,f(Y^{(n)}) \to \E f(B) \text{ as } n \to \infty.
\end{equation}
We see from Lemma \ref{lem:lift_equiv} that 
	\begin{equation}
	(X^{(n)}_t)_{t \geq 0} \Rightarrow (B_t^\T)_{t \geq 0} \quad\text{ if and only if }\quad(\ell[X^{(n)}_\cdot](t))_{t \geq 0} \Rightarrow (B_t)_{t \geq 0}.
	\end{equation}
So it is enough to prove that the convergence in law holds for $Z^{(n)}=\ell[X^{(n)}]$ with limiting law that of a usual $\R^d$-Brownian motion with diffusion constant $D$. 
We use \eqref{e:K_def} and \eqref{e:cn_T_other_def} to see that for any continuous and bounded functional $f$ on $\Ccal^{0}(\R^d)$
\begin{align}
	\E^\T_{\beta,n,r}f(Z^{(n)}) 
	&= \frac{1}{c_n^\T}\sum_{\omega^\T \in \Wcal_n^{\T}}f(\ell \circ x^{(n)}(\omega^\T))K[0,n] \nnb
	&= \frac{1}{c_n^\T}\sum_{\omega \in \Wcal_n}f(\ell \circ x^{(n)}\circ \hat \ell^{-1}(\omega)) K^{\T}[0,n] \nnb
	&= \frac{1}{c_n^\T}\sum_{\omega \in \Wcal_n}f(y^{(n)}(\omega)) K^{\T}[0,n]
\end{align}
where the last line follows from \eqref{e:lift_comp}. By definition,
\begin{align}
	\E_{\beta,n}f(Y^{(n)}) 
	&= \frac{1}{c_n}\sum_{\omega \in \Wcal_n}f(y^{(n)}(\omega))K[0,n].
\end{align}
Defining 
$\Ecal_n(f) = |\E^\T_{\beta,n,r}f(Z^{(n)}) -  \frac{c_n}{c_n^\T}\E_{\beta,n}f(Y^{(n)})|$ we have
\begin{align}
\label{e:ZminusB_bd}
	|\E^\T_{\beta,n,r}f(Z^{(n)})- \E f(B)| 
	\leq 
	\Ecal_n(f) 
	+ \left|1-\frac{c_n}{c_n^\T}\right| \, \E_{\beta,n}|f (Y^{(n)})| 
	+ |\E_{\beta,n} f(Y^{(n)})-\E f(B)| .
\end{align}
The third term in the right hand-side of \eqref{e:ZminusB_bd} goes to zero by assumption. The second term goes to zero as $n$ (and thus $r$) goes to infinity using the boundedness of $f$ and the fact that $c_n^\T \sim c_n$ from Theorem \ref{thm:dilute} when $n = o(V^{1/2})$. To control $\Ecal_n(f)$ we simply note that since $K^{\T}[0,n] \leq K[0,n]$ we have
\begin{align}
	\Ecal_n(f) 
	&\leq \frac{1}{c^\T_{n}} \sum_{\omega \in \Wcal_n}|f(y^{(n)}(\omega))|(K[0,n]-K^{\T}[0,n]) \nnb
	&\leq \|f\|_{\infty}\frac{c_n - c_n^\T}{c_n^{\T}} \to 0 \text{ as } r \to \infty
\end{align}
and the last line follows from \eqref{e:cnasy} and Theorem \ref{thm:dilute} since $\beta$ is sufficiently small and $n\to \infty$ with $n=o(V^{1/2})$.
\end{proof}
\begin{rk}
\label{rmk:degener_lim}
Following a similar proof as above, it is easy to check that if $\omega^\T \sim \beta$-WSAW$_{n,r}^\T$ with $n\ll r^2$ then any rescaling of $\omega^\T$ and $\T_r^d$ converges to $0$ in probability in the sense that for every $\eps>0$
\begin{equation}
	\lim_{\substack{n,r \to \infty\\ n \ll r^2}} \P_{\beta,n,r}^\T(r^{-1} \sup_{1 \leq k \leq n}|\omega^\T(k)|  > \eps) = 0.
\end{equation}
Indeed we have that for any $\eps >0$ and similarly as in the previous proof (although not making explicit the sums over walks) that
\begin{align}
	\P_{\beta,n,r}^\T(r^{-1} \sup_{1 \leq k \leq n}|\omega^\T(k)|  > \eps) 
	& \leq \P_{\beta,n,r}^\T(r^{-1} \sup_{1 \leq k \leq n}|\ell[\omega^\T](k)|  > \eps) \nnb
	&\leq \P_{\beta,n}(r^{-1} \sup_{1 \leq k \leq n}|\omega(k)|  > \eps) \nnb
	&= \P_{\beta,n}(n^{1/2}r^{-1} \sup_{1 \leq k \leq n}n^{-1/2}|\omega(k)|  > \eps)
\end{align}
and the proof follows from the hypothesis that $n \ll r^2$ and from tightness of $(\sup_{1 \leq k \leq n}n^{-1/2}|\omega(k)|)_n$ which is a consequence of the convergence in law of the $\Z^d$-WSAW to Brownian motion on $\Z^d$ when $d>4$.
\end{rk}
\subsection{Proof of Theorem \ref{thm:main}}
In order to conclude the proof of the main theorem it is enough by Lemma~\ref{lem:implication_conv} to show that \eqref{e:full_zd_conv} holds. We thus need to show that our rescaled (see \eqref{e:y_def}) $\Z^d$-WSAW  converges to an infinite Brownian path. This latter fact is not known. Indeed, the original $\Z^d$-scaling limit result presented in \cite[Theorem 6.1.8]{MS93} is different. It is shown there that a length $n$ $\beta$-WSAW $\omega$ rescaled by $\sqrt{n}$ converges in distribution to $(B_t)_{0 \leq t \leq 1}$. This last subsection thus gives a proof that 
\begin{equation}
\label{e:full_zd_conv_bis}
	(Y^{(n)}_t)_{t \geq 0} \Rightarrow (B_t)_{t \geq 0}.
\end{equation}
To reduce the proof further we use \cite[Proposition 16.6]{Kal21_2nd} which states that the convergence in law of continuous random processes indexed by $\R^+$ and taking values in some general metric space holds if and only if the convergence holds for the restrictions to any compact subset of $\R^+$. In fact $\R^+$ could be replaced by much more general spaces but we do not need the full generality of the proposition here. In our case, and with some further simplifications, this means that it is enough to prove the following proposition.
\begin{prop}
\label{prop:Y_to_B}
	Let $\beta$ be small enough. For $\omega \sim \beta$-{\rm WSAW}$_n$ on $\Z^d$ and $Y^{(n)} = y^{(n)}(\omega)$ we have that for all $T\in (0, \infty)$ 
	\begin{equation}
		(Y^{(n)})_{0 \leq t \leq T} \Rightarrow (B_{t})_{ 0 \leq t \leq T}.
	\end{equation}
where the convergence in law holds in the space $\Ccal_0^{0}([0,T],\R^d)$ instead of $\Ccal_0^{0}([0,\infty),\R^d)$.
\end{prop}

We thus need to prove that for all $T \in (0, \infty)$, the subwalk starting from $0$ of length $r^2 T$ of a length-$n$ $\Z^d$-WSAW converges to Brownian motion once rescaled by $r$. 
As usual, the proof has two steps, we need to show that the finite-dimensional distributions (fdd) of $Y^{(n)}$ converge to those of a Brownian motion with diffusion constant $D$ and then prove tightness of the sequence of processes. The proof requires only a small generalization of \cite[Chapter 6]{MS93} that we explain in the next two subsections. Before doing so we introduce the following notation which we use in both parts of the proof. We note that for $\omega \in \Wcal_{n}$
\begin{align}
\label{e:interp}
	Y^{(n)}_t(\omega) &= Y^{(n)}_{t_r}(\omega)+ \varphi_{r,n,t}(\omega),\\
	\varphi_{r,n,t}(\omega) &= (r^2t-\lfloor r^2t \rfloor) (Y^{(n)}_{t_r+r^{-2}}(\omega) - Y^{(n)}_{t_r}(\omega))
\end{align} with $t_r = \lfloor r^2t\rfloor/r^2$ and where $\varphi_{r,n,t}(\omega)$ satisfies
\begin{equation}
\label{e:varphi_neglig}
	|\varphi_{r,n,t}(\omega)| \leq \frac{1}{r} \quad \text{ uniformly in }\omega,t,T.
\end{equation}
\subsubsection{Finite-dimensional distributions}
Note that we can equivalently determine the fdd of $Y_{t}$ by replacing in the computations $Y_{t}$ by $Y_{t_r}$ by \eqref{e:interp}, \eqref{e:varphi_neglig} and the discussion along \cite[p.88--89]{Bill95}. 
Since the convergence of the fdd is implied by convergence of the characteristic function, the statement of convergence of the fdd takes the form of Lemma \ref{lem:fdd}. Also in the next lemma we let $k_n$ be any sequence going to infinity and satisfying $Tk_n\leq n$ for $n$ large enough. Here $k_n$ plays the role of $r^2$ and is typically $\ll n$ for our purposes but the following lemma does not require this condition.
\begin{lemma}[Finite-dimensional distributions of $Y_t^{(n)}(\omega)$]
\label{lem:fdd}
Let $d>4$, $\beta>0$ small enough and $T>0$.  For any increasing sequence $k = k_n$ such that $T k_n\leq n$ for all $n$ large enough and $k_n \to \infty$, and for any integer $N \geq 1$,  $0= t_0 < t_1< \cdots <t_N \leq T$ and $u_1,\cdots,u_N \in \R^d$, the following holds
\begin{equation}
\label{e:fdd_goal}
	\lim_{n \to \infty}\E_{\beta,n}\exp\Big[i\sum_{j=1}^{N}\frac{u_j}{\sqrt{D k_n}}\cdot\Big(\omega(\lfloor t_j k_n \rfloor)-\omega(\lfloor t_{j-1} k_n \rfloor)\Big)\Big] = \exp\Big(-\frac{1}{2d}\sum_{j=1}^N |u_j|^2(t_j-t_{j-1})\Big).
\end{equation}

\end{lemma}

The proof follows from an adaptation (and application) of \cite[Theorem 6.6.2]{MS93} that is given hereafter. While it is true that this theorem has been established for the SAW rather than the WSAW, it readily adapts to the easier case of the WSAW. Furthermore, the most technical part of the proof of the aforementioned result is the scaling limit of the endpoint \cite[Theorem 6.6.1]{MS93} which was proven in the original article \cite{BS85} this time for the WSAW model and we shall therefore use \cite[Theorem 6.6.2]{MS93} in the case $d>4$ and $\beta>0$ small enough. 

\begin{proof}[Proof of Lemma \ref{lem:fdd}]
Let $d>4$, $\beta>0$,  $T>0$ and $k = k_n$ such that $T k_n\leq n$ and $k_n \to \infty$ and fix arbitrary $N \geq 1$,  $0= t_0 < t_1< \cdots <t_N \leq T$ and $u_1,\cdots,u_N \in \R^d$.
For any vector $v$ we let $v^{(k)} := v/\sqrt{D k}$ and introduce further
\begin{equation}
	\Delta^{(k)}_{j}(\omega) = \omega(\lfloor t_jk \rfloor ) - \omega(\lfloor t_{j-1}k \rfloor) 
\end{equation}
where the dependency on $t$ is omitted for simplicity. Then  \eqref{e:fdd_goal} rewrites as
\begin{align}
\label{e:fdd_goal_2}
	\lim_{n \to \infty} c_n^{-1}M_n(u,t)
	&= \exp\Big(-\frac{1}{2d}\sum_{j=1}^N |u_j|^2(t_j-t_{j-1})\Big)	
\end{align}
with
\begin{align}
	M_n(u,t) &= 
	\sum_{\omega \in \Wcal_n} \exp\Big[i\sum_{j=1}^N u_j^{(n)} \cdot\Delta_j^{(k_n)}(\omega)\Big]K[0,n]\nonumber.
\end{align}

In order to prove \eqref{e:fdd_goal_2} we use a ``KJK expansion" to decouple the two parts of the walk before and after time $\lfloor k_nt_N\rfloor$. For this we appeal to Lemma 5.2.5 in \cite{MS93} where it is shown that for any integer $m \in [0,n]$ 
\begin{equation}
\label{e:KJK}
	K[0,n] = \sum_{I \ni m}K[0,I_1]J[I_1,I_2]K[I_2,n]
\end{equation}
where the sum is over intervals $I$ of the form $[I_1,I_2]$ where either $0 \leq I_1 < m < I_2 \leq b$ or $I_1 = I_2 = m$. We do not really need the definition of $J[s,t]=J[s,t](\omega)$ which can be found in \cite{BS85,MS93,Slad06} but rather only need to know that it is an interaction term (like $K[s,t]$) that only depends on $\omega$ between times $s$ and $t$ and satisfies
\begin{equation}
\label{e:pi_bound}
	\sum_{n=0}^\infty n z_c^n \sum_{\omega \in \Wcal_n} |J[0,n]| < \infty
\end{equation}
where $z_c = \mu^{-1}$ as was originally proven in \cite[Section 5]{BS85}. We refer to \cite[Theorem 6.2.9]{MS93} for a more concise and systematic treatment of this fact and other related \emph{diagrammatic estimates}.
The input of \eqref{e:KJK} gives 
\begin{align}
	M_n(u,t) 
	&= \sum_{I \ni \lfloor k_nt\rfloor} \sum_{\omega \in \Wcal_n}  \exp\Big[i\sum_{j=1}^N u_j^{(n)} \cdot\Delta_j^{(k_n)}(\omega)\Big]K[0,I_1]J[I_1,I_2]K[I_2,n].
\end{align}
We split the walk $\omega$ into three pieces from times $0$ to $I_1$, from $I_1$ to $I_2$ and from $I_2$ to $n$. We denote these walks $\omega_1,\omega_2$ and $\omega_3$ respectively and see from the above decoupling that 
\begin{align}
	M_n(u,t) &= \sum_{I \ni \lfloor k_nt\rfloor} \sum_{\substack{\omega_1 \in \Wcal_{I_1} \\ \omega_2 \in \Wcal_{I_2-I_1}}}\exp\Big[i\sum_{j=1}^N u_j^{(n)} \cdot\Delta_j^{(k_n)}(\omega_1 \circ \omega_2)\Big]K[0,I_1]J[0,I_2-I_1]c_{n-I_2}
\end{align}
where $\omega_1 \circ \omega_2$ is the the walk obtained by concatenation of $\omega_1$ and $\omega_2$.
We show that the main contribution to the above sum is given by those intervals $I$ such that $|I| = I_2-I_1 \leq b_n$ where $b_n$ ultimately needs to 
satisfy $b_n = o(k_n^{1/2})$ and $b_n \to \infty$. Consider first the 
case $|I| \leq b_n$ (resp. $|I| > b_n$) and denote by $M^{\leq}(u,n)$ (resp $M^{>}(u,n)$) the corresponding 
sum restricted to this range of $|I|$ such that $M_n(u,t) = M^{\leq}_n(u,t) + M_n^{>}(u,t)$. In this case for $n$ large enough we have $I_1 > k_nt_{N-1}$ and $I_1 \in [k_nt_N(1-{k_n}^{-1/2}),  k_nt_N]$ $I_1 \in [k_nt_N(1-{k_n}^{-1/2}),  k_nt_N]$ uniformly in $I \leq b_n$ which implies that $\Delta_j^{(k_n)}(\omega_1 \circ \omega_2) = \Delta_j^{(k_n)}(\omega_1)$ for all $1 \leq j \leq N-1$.   By adding and substracting $iu_N^{(n)}\cdot \omega_1(I_1)$ inside the exponential we see that \cite[Theorem 6.6.2]{MS93} applies (with $k_n$ in place of $n$) for the sum over $\omega_1$ and gives 
\begin{align}
	\sum_{\omega_1 \in \Wcal_{I_1}}
\exp\Big[i\sum_{j=1}^{N-1} u_j^{(n)} \cdot\Delta_j^{(k_n)}(\omega_1 ) + iu_N^{(n)} &\cdot(\omega_1(I_1) - \omega_1(\lfloor k_nt_{N-1}\rfloor))\Big]K[0,I_1] \nnb 
&= c_{I_1}\exp(-\frac{1}{2d}\sum_{j=1}^{N}|u_j|^2(t_j-t_{j-1}))(1+o(1))
\end{align}
and the $o(1)$ goes to zero as $n\to \infty$ uniformly in $|I| \leq b_n$. For the sum over $\omega_2$ we note that
\begin{equation}
\exp(iu_N^{(n)} \cdot\omega_2(\lfloor k_nt_N\rfloor-I_1) = 1+o(1)
\end{equation}
uniformly in $|I| \leq b_n$ so that overall
\begin{align}
M^{\leq}_n(u,t) 
&= (1+o(1))\exp(-\frac{1}{2d}\sum_{j=1}^{N}|u_j|^2(t_j-t_{j-1}))\sum_{\substack{I \ni \lfloor k_nt\rfloor \\ |I| \leq b_n}} c_{n-I_2}c_{I_1} \sum_{ \omega_2 \in \Wcal_{|I|}}
J[0,I_2-I_1].
\end{align}
By using \eqref{e:KJK} again and dividing both sides by $c_n$ this gives 
\begin{align}
\label{e:M_small_I}
c_n^{-1}M^{\leq}_n(u,t) 
%&= \exp(-\frac{1}{2d}\sum_{j=1}^{N}|u_j|^2(t_j-t_{j-1}))(1+o(1)) \nnb
%&- \exp(-\frac{1}{2d}\sum_{j=1}^{N}|u_j|^2(t_j-t_{j-1}))(1+o(1))\sum_{\substack{I \ni \lfloor k_nt\rfloor \\ |I| > b_n}} c_{n-I_2}c_{I_1} \sum_{ \omega_2 \in \Wcal_{|I|}}J[0,|I|] \\
&= \exp(-\frac{1}{2d}\sum_{j=1}^{N}|u_j|^2(t_j-t_{j-1}))(1+o(1)) -O(1)\sum_{\substack{I \ni \lfloor k_nt\rfloor \\ |I| > b_n}} \frac{c_{n-I_2}c_{I_1}}{c_n} \sum_{ \omega_2 \in \Wcal_{|I|}}J[0,|I|].
\end{align}
Similarly for $M_n^{>}(u,t)$, but treating the exponential in the summand simply as a $O(1)$, we obtain directly by \eqref{e:KJK} that 
\begin{equation}
\label{e:M_large_I}
	c_n^{-1}M_n^{>}(u,t) = O(1)\sum_{\substack{I \ni \lfloor k_nt\rfloor \\ |I| > b_n}} \frac{c_{n-I_2}c_{I_1}}{c_n}  \sum_{ \omega_2 \in \Wcal_{|I|}}J[0,|I|].
\end{equation}

By summing both contributions \eqref{e:M_small_I} and \eqref{e:M_large_I} we see that to conclude it is enough to show that $c_n^{-1}M_n^{>}(u,t) \to 0$ as $n\to \infty$. To see this we use the fact that $c_n = O(z_c^{-n})$ from \eqref{e:cnasy} to obtain
\begin{equation}
	c_n^{-1}M_n^{>}(u,t)
	= O(1)\sum_{\substack{I \ni \lfloor k_nt\rfloor \\ |I| > b_n}} z_c^{|I|}  \sum_{ \omega_2 \in \Wcal_{|I|}}J[0,|I|].
\end{equation}
Finally we see that $|I|$ being fixed there are at most $|I|+1$ ways to choose $I \ni \lfloor k_nt\rfloor$ so that
\begin{equation}
	c_n^{-1}M_n^{>}(u,t)
	= O(1)\sum_{\substack{|I| = b_n+1}}^\infty |I| z_c^{|I|}  \sum_{ \omega_2 \in \Wcal_{|I|}}J[0,|I|].
\end{equation}
The absolute bound \eqref{e:pi_bound}
achieves the proof since $b_n \to \infty$ and the sum is convergent.
\end{proof}

\subsubsection{Tightness}
We use the following criterion taken from \cite[Corollary 16.9]{Kal21_2nd} and adapted to our purposes to obtain tightness of random continuous processes under some moment condition. The proof of \eqref{e:tightness} itself resembles that of \cite{Slad88} where the proof of tightness for the ordinarily rescaled SAW originates. We note that the criterion used here is slightly simpler than the one used in \cite{Slad88}.
\begin{lemma}[Tightness of $Y^{(n)}(\omega)$] 
\label{lem:tightness}
Let $d>4$, $\beta>0$ small enough, $T>0$ and $\omega\sim \beta$-{\rm WSAW}$_n$ then the sequence $(Y^{(n)}_t(\omega))_{0 \leq t \leq T}$ is tight if there exists $A>0$ such that for all $0 \leq s < t \leq T$ and uniformly in $n$
\begin{equation}
\label{e:tightness}
	\E_{\beta,n} |Y^{(n)}_{t}(\omega)-Y^{(n)}_{s}(\omega)|^2 \leq A |s-t|.
\end{equation}
\end{lemma}
We now prove that \eqref{e:tightness} holds therefore completing the proof of Proposition \ref{prop:Y_to_B} by combining Lemma \ref{lem:fdd} and Lemma \ref{lem:tightness}. Let $0\leq s<t \leq T$ and suppose first that $t_r = s_r$ where we recall that for any real $u$, we let $u_r = \lfloor r^2 u\rfloor / r^2$. Then we have by \eqref{e:interp}--\eqref{e:varphi_neglig} that for any walk $\omega \in \Wcal_n$
\begin{equation}
	|Y^{(n)}_t(\omega)-Y^{(n)}_s(\omega)| = (r^2t-r^2s)|Y^{(n)}_{t_r+r^{-2}}(\omega)-Y^{(n)}_{t_r}(\omega)| \leq r(t-s).
\end{equation}
and since $t_r=t_s$ implies $r \leq |t-s|^{-1/2}$ the result follows with $A = 1$ by squaring the above equation and taking the expectation on both sides. 

Now we suppose that $t_r \neq s_r$, then by \eqref{e:interp}--\eqref{e:varphi_neglig} again, we can write
\begin{align}
	|Y^{(n)}_t(\omega)-Y^{(n)}_s(\omega)| 
	&\leq |Y^{(n)}_{t_r}(\omega)-Y^{(n)}_{s_r}(\omega)| + |\varphi_{n,t,r}(\omega)|+|\varphi_{n,s,r}(\omega)| \nnb
	&\leq |Y^{(n)}_{t_r}(\omega)-Y^{(n)}_{s_r}(\omega)| + \frac{2}{r}.
\end{align}
Since $t_r \neq s_r$ implies $r \geq |t-s|^{-1/2}$ we further obtain using $(a+b)^2 \leq 2(a^2+b^2)$ that
\begin{align}
	\E_{\beta,n} |Y^{(n)}_{t}(\omega)-Y^{(n)}_{s}(\omega)|^2 
	&\leq 2\E_{\beta,n} |Y^{(n)}_{t_r}(\omega)-Y^{(n)}_{s_r}(\omega)|^2 + 8|t-s|.
\end{align}
And thus to conclude with the proof of \eqref{e:tightness} we see that it is enough to show that 
\begin{equation}
\label{e:tightness_suffic}
	\E_{\beta,n} |Y^{(n)}_{t_r}(\omega)-Y^{(n)}_{s_r}(\omega)|^2 \leq A|t-s|.
\end{equation} 
By definition the left-hand side of \eqref{e:tightness_suffic} is equal to
\begin{align}
\label{e:tightness_suffic_2}
	\frac{1}{r^2 c_n}\sum_{\omega \in \Wcal_n}|\omega(\lfloor r^2t\rfloor)-\omega(\lfloor r^2s\rfloor)|^2K[0,n].
\end{align}
Using the fact that $K[0,n] \leq K[0,\lfloor r^2 s \rfloor]\,K[\lfloor r^2 s \rfloor,\lfloor r^2 t \rfloor]\,K[\lfloor r^2 t \rfloor,n]$, decomposing $\omega$ into three subwalks corresponding to these time intervals and using that ($A$ may change from line to line)
\begin{equation}
	c^{-1}_n \leq A c^{-1}_{\lfloor r^{2}s\rfloor}c^{-1}_{\lfloor r^{2}t\rfloor-\lfloor r^2 s\rfloor} c^{-1}_{n-\lfloor r^2 t\rfloor} 
\end{equation}
from \cite[Theorem 6.1.1(a)]{MS93} we get that \eqref{e:tightness_suffic_2} satisfies
\begin{equation}
	\frac{1}{r^2 c_n}\sum_{\omega \in \Wcal_n}|\omega(\lfloor r^2t\rfloor)-\omega(\lfloor r^2s\rfloor)|^2K[0,n] 
	\leq \frac{A}{r^2}\E_{\lfloor r^2t\rfloor-\lfloor r^2s\rfloor,\beta}|\omega(\lfloor r^2t\rfloor-\lfloor r^2s\rfloor)|^2.
\end{equation}
which by the bound on the mean square displacement in \eqref{e:mean_square_disp} gets bounded further by $A |t-s|$
and thus achieves the proof.
\section*{Acknowledgements}

This work was supported in part by NSERC of Canada.
We thank Gordon Slade for introducing this topic to us, for our collaboration on \cite{MS22} where the idea of this paper idea arose as a natural continuation as well as for useful comments on an earlier version of this paper.

\bibliography{lg}
\bibliographystyle{plain}

\end{document}